\documentclass[11pt]{amsart}
\usepackage[all]{xy}

\usepackage{amsmath}
\usepackage{amsfonts}
\usepackage{amssymb}
\usepackage{mathrsfs}

\usepackage{verbatim}
\usepackage{url}

\DeclareFontEncoding{OT2}{}{} 

\def\ep{{\e\prime}}
\def\epp{{\e\prime\prime}}
\def\eppp{{\e\prime\prime\prime}}

\def\Hom{{\rm Hom}}

\def\bh{{\mathbb H}}

\def\bz{{\mathbb Z}\,}
\def\bq{{\mathbb Q}}
\def\bg{{\mathbb G}}
\def\spec{{\rm{Spec}}\,}

\def\der{{\le\rm{der}}}
\def\rad{{\e{\rm{rad}}}}

\def\be{\kern -.1em}
\def\lbe{\kern -.025em}

\def\hom{{\rm{Hom}\e}}
\def\rhom{{\rm{RHom}}\e}
\def\krn{{\rm{Ker}}\e }

\def\cok{{\rm{Coker}}}

\def\tor{{\e\rm{tor}}}

\def\Gtil{{\widetilde{G}}}

\def\id{{\rm id}}
\def\sR{{\mathscr{R}}}
\newcommand{\G}{{\mathbb{G}}}

\def\Ttil{{\widetilde{T}}}
\newcommand{\vk}{\varkappa}
\newcommand{\into}{\hookrightarrow}

\newcommand{\isoto}{\overset{\!\sim}{\to}}
\def\labelto#1{\smash{\mathop{\longrightarrow}\limits^{#1}}}

\def\s{\mathscr }
\def\ra{\rightarrow}
\def\e{\kern 0.08em}
\def\le{\kern 0.03em}
\def\ng{\kern -0.04em}

\def\krn{{\rm{Ker}}\,}
\def\cok{{\rm{Coker}}\,}

\def\C{{\mathbb{C}}}
\def\et{{\text{\'et}}}
\def\top{{\rm top}}

\newtheorem{lemma}{Lemma}[section]
\newtheorem{theorem}[lemma]{Theorem}
\newtheorem{corollary}[lemma]{Corollary}
\newtheorem{proposition}[lemma]{Proposition}
\theoremstyle{definition}
\newtheorem{definition}[lemma]{Definition}
\newtheorem{construction}[lemma]{Construction}
\theoremstyle{remark}
\newtheorem{remark}[lemma]{Remark}

\def\bM{\begin{commentM}}
\def\eM{\end{commentM}\ }

\begin{document}

\title[Algebraic fundamental group]
{The algebraic fundamental group\\  of a reductive group scheme\\
over an arbitrary base scheme}

\author{ Mikhail Borovoi}
\address{Borovoi: Raymond and Beverly Sackler School of Mathematical Sciences,
Tel Aviv University, 69978 Tel Aviv, Israel}
\email{borovoi@post.tau.ac.il}

\author{Cristian D. Gonz\'alez-Avil\'es}
\address{Gonz\'alez-Avil\'es: Departamento de Matem\'aticas, Universidad de La Serena, Cisternas 1200, La Serena 1700000, Chile}
\email{cgonzalez@userena.cl}


\begin{abstract}
We define the algebraic fundamental group $\pi_{1}(G)$ of
a reductive group scheme $G$ over an arbitrary non-empty base scheme and show that the resulting functor $G\mapsto\pi_{1}(G)$ is exact.
\end{abstract}

\maketitle

\section{Introduction}

If $G$ is a (connected) reductive algebraic group over a field $k$ of characteristic 0 and $T$ is a maximal $k$-torus of $G$,
the algebraic fundamental group $\pi_{1}(G,T)$ of the pair $(G,T)$ was defined by the first-named author \cite{bor}
and shown there to be
independent (up to a canonical isomorphism) of the choice of $T$
and useful in the study of the first Galois cohomology set of $G$.
See Definition \ref{g,t} below for a generalization of the original definition of $\pi_{1}(G,T)$.
Independently, and at about the same time, Merkurjev \cite[\S10.1]{Merkurjev} defined the algebraic fundamental group of $G$ over an arbitrary field.
Later, Colliot-Th\'el\`ene \cite[Proposition-Definition 6.1]{ct} defined the algebraic fundamental group $\pi_{1}(G)$ of $G$
in terms of a flasque resolution of $G$, showed that his definition was independent (up to a canonical isomorphism)
of the choice of the resolution, and established the existence of a canonical isomorphism $\pi_{1}(G)\simeq\pi_{1}(G,T)$,
see \cite[Proposition A.2]{ct}.
Recall that a \emph{flasque resolution of $G$} is a central extension
$$
1\to F\to H\to G\to 1
$$
where the derived group $H^{\der}$ of $H$ is simply connected, $H^{\tor}:=H/H^{\der}$
is a quasi-trivial $k$-torus, and $F$ is a \emph{flasque} $k$-torus,
i.e., the group of cocharacters of $F$ is an $H^{1}$-trivial Galois module.
It turns out that flasque resolutions of reductive group schemes exist over bases that are more general than spectra of fields,
and the second-named author has used such resolutions to generalize Colliot-Th\'el\`ene's definition
of $\pi_{1}(G)$ to reductive group schemes $G$ over any non-empty, reduced, connected,
locally Noetherian and geometrically unibranch scheme. See \cite[Definition 3.7]{GA-flasque}.

In the present paper we extend  the definition  of  \cite{GA-flasque}
to reductive group schemes $G$ over an \emph{arbitrary} non-empty scheme.
Since flasque resolutions are not available in this general setting (see \cite[Remark 2.3]{GA-flasque}),
we shall use instead $t$-resolutions, which exist over any non-empty base scheme $S$.
A \emph{$t$-resolution} of $G$ is a central extension
$$
1\to T\to H\to G\to 1,
$$
where $T$ is an $S$-torus and $H$ is a reductive $S$-group scheme such that the derived group $H^{\der}$ is simply connected.
Since a flasque resolution is a particular type of $t$-resolution,
the definition of $\pi_{1}(G)$ given here (Definition \ref{def:pi_1}) does indeed extend
the definition
of the second-named author  \cite{GA-flasque}.
Further, since the choice of a maximal $S$-torus of $G$ (when one exists) canonically determines a $t$-resolution of $G$
(see Lemma \ref{c:t-special}), our Definition \ref{def:pi_1} turns out to be a common generalization
of the definitions of \cite{bor}
and of  \cite{ct} and \cite{GA-flasque}.

Once the general definition of $\pi_{1}(G)$ is in place,
we proceed to study some of the basic properties of the resulting functor $G\mapsto\pi_{1}(G)$,
culminating in a proof of its exactness (Theorem \ref{thm:pi1-exact}).
We give, in fact, two proofs of Theorem \ref{thm:pi1-exact}, the second of which makes use of the \'etale-local existence of maximal tori
in reductive $S$-group schemes and generalizes \cite[proof of Lemma 3.7]{BKG}.
In the final section of the paper we use $t$-resolutions to relate the (flat) abelian cohomology of $G$ over $S$
introduced in \cite{ga2} to the cohomology of $S$-tori, thereby generalizing \cite[\S4]{GA-flasque}.

\begin{remark}
Let $G$ be a (connected) reductive group over the field of complex numbers $\C$.
Here we comment on the interrelation between the algebraic fundamental group $\pi_1(G)$, the topological fundamental group $\pi_1^\top(G(\C))$,
and the \'etale fundamental group $\pi_1^\et(G)$.
By \cite[Prop.~1.11]{bor} the algebraic fundamental group $\pi_1(G)$ is canonically isomorphic to the group
\[
\pi_1^\top(G(\C))(-1):= \Hom(\pi_1^\top(\C^\times), \pi_1^\top(G(\C)).
\]
It follows that  if $G$ is a  reductive $k$-group $G$ over an algebraically closed field $k$ of characteristic zero,
then the profinite completion of $\pi_1(G)$ is canonically isomorphic to the group
\[
\pi_1^\et(G)(-1):= \Hom_{\rm cont}(\pi_1^\et(\G_{m,k}), \pi_1^\et(G)).
\]
where $\Hom_{\rm cont}$ denotes the group of continuous homomorphisms and  $\G_{m,k}$ denotes the multiplicative group over $k$.
See \cite{BD} for details and for a generalization of the algebraic fundamental group $\pi_1(G)$
to arbitrary homogeneous spaces of connected linear algebraic groups.
\end{remark}

\bigskip\noindent

{\bf Notation and terminology.}
Throughout this paper, $S$ denotes a non-empty scheme. An \emph{$S$-torus} is an $S$-group scheme
which is fpqc-locally isomorphic to a group of the form $\bg_{m,S}^{n}$ for some integer $n\geq 0$ \cite[Exp.~IX, Definition 1.3]{sga3}.
An $S$-torus is affine, smooth and of finite presentation over $S$ \cite[Exp.~IX, Proposition 2.1(a), (b) and (e)]{sga3}.
An $S$-group scheme $G$ is called
\emph{reductive} (respectively, \emph{semisimple}, \emph{simply connected}) if it is affine and smooth over $S$ and its
geometric fibers are \emph{connected} reductive (respectively, semisimple,
simply connected) algebraic groups \cite[Exp.~XIX, Definition 2.7]{sga3}.
An $S$-torus is reductive, and any reductive $S$-group scheme is of finite presentation over $S$ \cite[Exp.~XIX, 2.1]{sga3}.
Now, if $G$ is a reductive group scheme over $S$,  $\rad(G)$ will denote the radical of $G$,
i.e., the identity component of the center $Z(G)$ of $G$. Further, $G^{\der}$ will denote the derived group of $G$.
Thus $G^{\der}$ is a normal semisimple subgroup scheme of $G$ and $G^{\tor}:=G/G^{\der}$ is the largest quotient of $G$ which is an $S$-torus.
We shall write $\Gtil$ for the simply connected central cover of $G^{\der}$
and $\mu:=\krn\!\big[\e\Gtil\to G^{\der}\e\big]$ for the fundamental group of $G^{\der}$.
See \cite[\S2]{GA-flasque} for the existence and basic properties of $\widetilde{G}$.
There exists a canonical homomorphism $\partial\colon \Gtil\to G$ which factors as
$\Gtil\twoheadrightarrow G^{\der}\hookrightarrow G$. In particular, $\krn\partial=\mu$ and $\cok\partial=G^{\tor}$.

If $X$ is a (commutative) finitely generated twisted constant $S$-group scheme \cite[Exp.~X, Definition 5.1]{sga3},
then $X$ is quasi-isotrivial, i.e., there exists a surjective \'etale morphism $S^{\e\prime}\ra S$ such that $X\times_{S}S^{\e\prime}$ is constant.
Further, the functors
$$
X\mapsto X^{*}:=\underline{\hom}_{\, S\text{-gr}}(X,\bg_{m,S})\quad \text{and}\quad
M\mapsto M^{*}:=\underline{\hom}_{\, S\text{-gr}}(M,\bg_{m,S})
$$
are mutually quasi-inverse anti-equivalences
between the categories of finitely generated twisted constant $S$-group schemes and $S$-group schemes of finite type
and of multiplicative type\footnote{Although \cite[Exp.~IX, Definition 1.4]{sga3} allows for groups of multiplicative type
which may not be of finite type over $S$, such groups will play no role in this paper.}
  \cite[Exp.~X, Corollary 5.9]{sga3}.
Further, $M\to M^{*}$ and $X\to X^{*}$ are exact functors (see \cite[Exp.~VIII, Theorem 3.1]{sga3} and
use faithfully flat descent). If $G$ is a reductive $S$-group scheme, its group of characters $G^{\le*}$ equals $(G^{\tor})^{\le*}$
(see \cite[Exp.~XXII, proof of Theorem 6.2.1(i)]{sga3}). Now, if $T$ is an $S$-torus, the
functor $\underline{\hom}_{\, S\text{-gr}}(\bg_{m,S},T\e)$ is
represented by a (free and finitely generated) twisted constant $S$-group
scheme which is denoted by $T_{*}$ and called the \emph{group of cocharacters of $\,T$}
(see \cite[Exp.~X, Corollary 4.5 and Theorem 5.6]{sga3}).
There exists a canonical isomorphism of free and finitely generated twisted constant $S$-group schemes
\begin{equation}\label{char-cochar}
T^{*}\simeq(T_{*})^{\vee}:=\hom_{S\lbe\text{-gr}}(T_{*},\bz_{\! S}).
\end{equation}

A sequence
\begin{equation}\label{e:exact}
0\to T\to H\to G\to 0
\end{equation}
of reductive $S$-group schemes and $S$-homomorphisms is called \emph{exact}
if it is exact as a sequence of sheaves for the fppf topology on $S$.
In this case the sequence \eqref{e:exact} will be called \emph{an extension of $G$ by $T$}.

\smallskip

If $G$ is a reductive $S$-group scheme, the identity homomorphism $G\to G$ will be denoted $\id_{G}$. Further, if $T$ is an $S$-torus,
the inversion automorphism  $T\to T$ will be denoted ${\rm{inv}}_{T}$.

\smallskip

\section{Definition of $\pi_1$}

\begin{definition}\label{def t-res}
Let $G$ be a reductive $S$-group scheme. A {\em  $t$-resolution of $G$} is a central extension
\begin{equation}\label{t-res}
1\to T\to H\to G\to 1,
\end{equation}
where $T$ is an $S$-torus and $H$ is a reductive $S$-group scheme
such that $H^{\der}$ is simply connected.
\end{definition}

\begin{proposition}\label{ex-t-resol} Every reductive $S$-group scheme admits a $t$-resolution.
\end{proposition}
\begin{proof} By \cite[Exp.~XXII, 6.2.3]{sga3}, the product in $G$ defines a
faithfully flat homomorphism $\rad(G\e)\times_{S}\e G^{\der}\ra G$ which induces
a faithfully flat homomorphism $\rad(G\e)\times_{S}\widetilde{ G }\ra G $. Let $\mu_{1}=\ker[\rad(G\e)\times_{S}\widetilde{ G }\ra G\e ]$,
which is a finite $S$-group scheme of multiplicative type contained in the center of $\rad(G)\times_S \Gtil$
(see \cite[proof of Proposition 3.2, p.~9]{GA-flasque}).  By \cite[Proposition B.3.8]{cnrd},
there exist an $S$-torus $T$ and a closed immersion $\psi\colon \mu_{1}\into T$. Let
$H$ be the pushout of $\varphi\colon\mu_{1}\hookrightarrow
\rad(G)\times_{S}\widetilde{G}$ and $\psi\colon\mu_{1}\hookrightarrow T$, i.e., the cokernel of the central embedding
\begin{equation}\label{e:H0-push}
(\varphi,{\rm{inv}}_{T}\be\circ\be \psi\e)_{S}\colon\mu_{1} \into \big(\rad(G)\times_{S}\Gtil\e\big)\times_{S}  T.
\end{equation}
Then $H$ is a reductive $S$-group scheme, cf.~\cite[Exp.~XXII, Corollary 4.3.2]{sga3}, which fits into an exact sequence
\[
1\to T\to H\to G\to 1,
\]
where $T$ is central in $H$. Now, as in \cite[proof of Proposition-Definition~3.1]{ct} and \cite[proof of Proposition~3.2, p.~10]{GA-flasque},
there exists an embedding of $\Gtil$ into $H$ which identifies $\Gtil$ with $H^{\der}$.
Thus $H^{\der}$ is simply connected, which completes the proof.
\end{proof}

As in \cite[p.~93]{ct} and \cite[(3.3)]{GA-flasque}, a $t$-resolution
\[
1\to T\to H\to G\to 1\tag{$\sR\e$}
\]
induces a ``fundamental diagram''
\[
\xymatrix{ &1\ar[d] &1\ar[d] &1\ar[d] &\\
1\ar[r] &\mu\ar[r]\ar[d] &\Gtil\ar[r]\ar[d] & G^{\der}\ar[r]\ar[d] &1\\
1\ar[r] & T\ar[r]\ar[d] & H \ar[r]\ar[d] & G \ar[r]\ar[d] &1\\
1\ar[r] &M\ar[r]\ar[d] &R\ar[r]\ar[d] & G^{\tor}\ar[r]\ar[d] &1,\\
&1 &1 &1 & }
\]
where $M=T/\mu$ and $R=H^{\tor}$. This diagram induces, in turn, a canonical isomorphism in the derived category
\begin{equation}\label{q-iso}
\big(Z\big(\Gtil\e\big)\!\be\overset{\partial_{Z}}
\longrightarrow\! Z(G)\e\big)\approx(T\be\ra\be R\e)
\end{equation}
(cf. \cite[Proposition 3.4]{GA-flasque}) and a canonical exact sequence
\begin{equation}\label{fund-diag-seq}
1\to\mu\to T\to R\to G^{\tor}\to 1,
\end{equation}
where $\mu$ is the fundamental group of $G^{\der}$. Since $\mu$ is finite, \eqref{fund-diag-seq}
shows that the induced homomorphism  $T_{\be *}\ra R_{*}$ is injective.
Set
\begin{equation}\label{def: p1-resol}
\pi_1(\sR\e)=\cok\! \big[\e T_{*}\to R_{*}\e\big].
\end{equation}
Thus there exists an exact sequence of (\'etale, finitely generated) twisted constant $S$-group schemes
\begin{equation}\label{t-r-p1}
1\ra T _{\be *}\ra R _{*}\ra\pi_1(\sR\e)\ra 1.
\end{equation}
Set
\[
\mu(-1):=\hom_{S\lbe\text{-gr}}(\e\mu^{*},(\bq/\bz\be)_{S}).
\]

\begin{proposition} \label{mu-p1-gtor}  A $t$-resolution $\sR$  of a reductive $S$-group scheme $G$ induces an
exact sequence of finitely generated twisted constant $S$-group schemes
\[
1\to\mu(-1)\to\pi_1(\sR\e)\to(G ^{\tor})_{*}\to 1.
\]
\end{proposition}
\begin{proof} The proof is  similar to that of \cite[Proposition 6.4]{ct}, using \eqref{fund-diag-seq}.
\end{proof}

\begin{definition}\label{def morph-resol}
Let $G$ be a reductive $S$-group scheme and let
\begin{gather}
 1\to T^{\e\prime}\to H^{\e\prime}\to G\to 1\tag{$\sR^{\e\prime}\e$}\\
 1\to T\to H\to G\to 1\tag{$\sR\e$}
\end{gather}
be two $t$-resolutions of $G$. A \emph{morphism from $\sR^{\e\prime}$ to $\sR$}, written $\phi\colon
\sR^{\e\prime}\to \sR$, is a commutative diagram
\begin{equation}\label{e:morphism}
\xymatrix{ 1\ar[r]   &T^{\e\prime}\ar[r]\ar[d]^{\phi_{\le T}}
&H^{\e\prime}\ar[r]\ar[d]^{\phi_{H}}
&G\ar[r]\ar[d]^{\id_{G}}  &1\\
1\ar[r]   &T\ar[r]         &H\ar[r]         &G\ar[r] &1, }
\end{equation}
where $\phi_{\e T}$ and $\phi_{H}$ are $S$-homomorphisms.
Note that, if $R^{\e\prime}=(H^{\e\prime})^{\tor}$ and $R=H^{\tor}$,
then $\phi_{H}$ induces an $S$-homomorphism $\phi_{R}\colon R^{\e\prime}\to R$.

We shall say that a $t$-resolution $\sR^{\e\prime}$ of $G$ {\em dominates} another $t$-resolution $\sR$ of $G$
if there exists a morphism $\sR^{\e\prime}\to \sR$.
\end{definition}

The following lemma is well-known.

\def\Cyl{{\rm Cyl}}
\begin{lemma}\label{l:quasi}
A morphism of complexes $f\colon P\to Q$ in an abelian category is a quasi-isomorphism if and only its cone $C(f)$ is acyclic
(i.e., has trivial cohomology).
\end{lemma}
\begin{proof}
 By \cite[Lemma III.3.3]{GM} there exists a short exact sequence of complexes
\begin{equation}\label{e:cyl-cone}
0\to P\to \Cyl(f)\to C(f)\to 0,
\end{equation}
where $\Cyl(f)$ is the cylinder of $f$.
Further, the complex $\Cyl(f)$ is canonically isomorphic to $Q$ in the derived category.
Now the short exact sequence \eqref{e:cyl-cone} induces a cohomology exact sequence
$$
\dots\to H^i(P)\to H^i(Q)\to H^i(C(f))\to H^{i+1}(P)\to\dots
$$
from which the lemma is immediate.
\end{proof}

\begin{lemma}\label{lem:Bernstein}
Let $g\colon C\to D$ be a quasi-isomorphism of bounded complexes of split $S$-tori.
Then the induced morphism of complexes of cocharacter $S$-group schemes
$g_*\colon C_*\to D_*$ is a quasi-isomorphism.
\end{lemma}
\begin{proof}
Since the assertion is local in the \'etale topology, we may and do assume that $S$ is connected.
The given quasi-isomorphism induces a quasi-isomorphism $g^*\colon D^*\to C^*$ of bounded complexes
of free and finitely generated constant $S$-group schemes.
Thus, by \eqref{char-cochar}, it suffices to check that the functor $X\mapsto X^{\vee}$ on the category of bounded complexes
of free and finitely generated constant $S$-group schemes preserves quasi-isomorphisms. We thank Joseph Bernstein for the following argument.
By Lemma \ref{l:quasi} a morphism $f\colon P\to Q$ of bounded complexes in the (abelian) category
of finitely generated constant $S$-group schemes
is a quasi-isomorphism if and only if its cone $C(f)$ is acyclic.
Now,  if $f\colon P\to Q$ is a quasi-isomorphism and $P$ and $Q$
are bounded complexes of \emph{free} and finitely generated constant $S$-group schemes, then $C(f)$ is an acyclic complex
of free and finitely generated constant $S$-group schemes. We see immediately that the dual complex
$$
C(f)^\vee=C\big(f^{\vee}\e\big)\big[-1\big]
$$
is acyclic, whence $f^\vee$ is a quasi-isomorphism by Lemma \ref{l:quasi}.
\end{proof}

\begin{lemma}\label{lem:induced-hom}
Let $G$ be a reductive $S$-group scheme and let $\sR^{\e\prime}$ be a $t$-resolution
of $G$ which dominates another $t$-resolution $\sR$ of $G$.
Then a morphism of $t$-resolutions $\phi\colon\sR^{\e\prime}\to\sR\e$
induces an isomorphism of finitely generated twisted constant $S$-group schemes
$\pi_{1}(\phi)\colon \pi_1(\sR^{\e\prime}\e)\isoto \pi_1(\sR\e)$ which is independent of the choice of $\phi$.
\end{lemma}
\begin{proof} Let  $\sR^{\e\prime}\colon 1\to T^{\e\prime}\to H^{\e\prime}\to G\to 1$ and $\sR\colon 1\to T\to H\to G\to 1$
be the given $t$-resolutions of $G$, as in Definition \ref{def morph-resol}, and set $R=H^{\tor}$ and $R^{\e\prime}=(H^{\e\prime})^{\tor}$.
Since the assertion is local in the \'etale topology,
we may and do assume that the tori $T,T^{\e\prime},R$ and $R^{\e\prime}$ are split and that $S$ is connected.
From \eqref{fund-diag-seq} we see that the morphism of complexes of split tori (in degrees 0 and 1)
$$
(\phi_{\e T},\phi_{\e R})\colon\big(T^{\e\prime}\to R^{\e\prime}\e\big)\to (T\to R\e)
$$
is a quasi-isomorphism. Now by Lemma \ref{lem:Bernstein},
\[
\pi_{1}(\phi):=H^{1}(
(\phi_{\e T},\phi_{\e R})_*)\colon \pi_1(\sR^{\e\prime}\e)\isoto \pi_1(\sR\e)
\]
is an isomorphism. In order to show that this isomorphism does not depend on the choice of $\phi$,
assume that $\psi\colon\sR^{\e\prime}\to\sR\e$ is another morphism of $t$-resolutions.
It is clear from  diagram  \eqref{e:morphism}  that $\psi_{H}$ differs from $\phi_{H}$
by some homomorphism $H^{\e\prime}\to T$ which factors through $R^{\e\prime}=(H^{\e\prime})^{\tor}$.
It follows that the induced homomorphisms $(\psi_{R})_{*},(\phi_{R})_{*}\colon R^{\e\prime}_{*}\to R_{*}$
differ by a homomorphism which factors through $T_{*}$. Consequently, the induced homomorphisms
\[
\pi_{1}(\phi), \pi_{1}(\psi)\colon \cok[T^{\e\prime}_{*}\to R^{\e\prime}_{*}]\to \cok[T_{*}\to R_{*}]
\]
coincide.
\end{proof}

\begin{proposition}\label{prop:Kottwitz}
Let $\vk\colon G_1\to G_2$ be a homomorphism of reductive $S$-group
schemes and let
\begin{gather}
 1\to T_1\to H_1\to G_1\to 1\tag{$\sR_1$}\\
 1\to T_2\to H_2\to G_2\to 1\tag{$\sR_2$}
\end{gather}
be $t$-resolutions of $G_{1}$ and $G_{2}$, respectively. Then there exists an exact commutative diagram
\begin{equation}\label{diag:Kottwitz}
\xymatrix{
1 \ar[r]  &T_1\ar[r]              &H_1\ar[r]                &G_1\ar[r]                       &1\\
1 \ar[r]  &T^{\e\prime}_1\ar[r]\ar[u]\ar[d]
&H^{\e\prime}_1\ar[r]\ar[u]\ar[d]
&G_1\ar[r]\ar[u]_{\id_{G}}\ar[d]^\vk  &1\\
1 \ar[r]  &T_2\ar[r]              &H_2\ar[r] &G_2\ar[r] &1,
}
\end{equation}
where the middle row is a $t$-resolution of $\e G_1$.
\end{proposition}
\begin{proof} We follow an idea of Kottwitz \cite[Proof of Lemma 2.4.4]{kott}.
Let $H^{\e\prime}_1=H_1\times_{G_2} H_2$, where the morphism $H_1\to G_2$ is the composition $H_1\to G_1\overset{\vk}\to G_2$.
Clearly, there are canonical morphisms
$H^{\e\prime}_1\to H_1$ and $H^{\e\prime}_1\to H_2$. Now, since $H_2\to G_2$ is faithfully flat, so also is $H^{\e\prime}_1\to H_1$.
Consequently the composition
$H^{\e\prime}_1\to H_1\to G_1$ is faithfully flat as well. Let $T^{\e\prime}_1$ denote its kernel,
i.e., $T^{\e\prime}_1=S\times_{G_{\e 1}}\be H^{\e\prime}_1$. Then
\[
T^{\e\prime}_1=(S\times_{G_{\e 1}}\be H_1)\times_{G_2}H_2=T_1\times_S(S\times_{G_2}H_2)=T_1\times_{S}T_2,
\]
which is an $S$-torus. The existence of diagram \eqref{diag:Kottwitz} is now clear. Further, since $T_i$ is central in $H_i$ ($i=1,2$),
$T^{\e\prime}_1=T_1\times_{S}T_2$ is central in $H^{\e\prime}_1=H_1\times_{G_2} H_2$. The $S$-group scheme $H^{\e\prime}_1$ is
affine and smooth over $S$ and has connected reductive fibers, i.e.,
is a reductive $S$-group scheme. Further, the faithfully flat morphism $H^{\e\prime}_1\to G_1$
induces a surjection $(H^{\e\prime}_1)^{\der}\to G_{1}^{\der}$ with (central) kernel $T^{\e\prime}_1\cap (H^{\e\prime}_1)^{\der}$.
Since $(H^{\e\prime}_1)^{\der}$ is semisimple, the last map is in fact a central isogeny. Consequently,
$(H^{\e\prime}_1)^{\der}\to H_1^{\der}=\widetilde{G}_1$ is a central isogeny as well,
whence $(H^{\e\prime}_1)^{\der}=\widetilde{G}_1$ is simply connected.
Thus the middle row of \eqref{diag:Kottwitz} is indeed a $t$-resolution of $G_1$.
\end{proof}

\begin{corollary}\label{cor:domination}
Let $\sR_1$ and $\sR_2$ be two $t$-resolutions of a
reductive $S$-group scheme $G$. Then there exists a $t$-resolution
$\sR_{\e 3}$ of $\e G\e$ which dominates both $\sR_1$ and $\sR_2$.
\end{corollary}

\begin{proof} This is immediate from Proposition \ref{prop:Kottwitz} (with $G_1=G_2=G$ and $\vk=\id_{G}$ there).
\end{proof}

\begin{lemma}\label{lem:can-isom}
Let $\sR_1$ and $\sR_2$ be two $t$-resolutions of a
reductive $S$-group scheme $G$. Then there exists a canonical isomorphism of finitely generated twisted constant $S$-group schemes
$\pi_{1}(\sR_1\e)\cong\pi_1(\sR_2\e)$.
\end{lemma}
\begin{proof} By Corollary \ref{cor:domination}, there exists a $t$-resolution
$\sR_{\e 3}$ of $\e G\e$ and morphisms of resolutions $\sR_{\e 3}\to \sR_{\e 1}$ and $\sR_{\e 3}\to \sR_{\e 2}$.
Thus, Lemma \ref{lem:induced-hom} gives a composite isomorphism $\psi_{\sR_3}\colon\pi_1(\sR_1\e)\isoto \pi_1(\sR_3\e)\isoto
\pi_1(\sR_2\e)$. Let
$\sR_{4}$ be another $t$-resolution of $G$ which dominates both $\sR_1$ and
$\sR_2$ and let $\psi_{\sR_4}\colon\pi_1(\sR_1\e)\isoto \pi_1(\sR_4\e)\isoto
\pi_1(\sR_2\e)$ be the corresponding composite isomorphism. There exists a $t$-resolution $\sR_{5}$
which dominates both $\sR_3$ and $\sR_4$.
Then $\sR_5$ dominates $\sR_1$ and $\sR_2$ and we obtain a composite isomorphism
$\psi_{\sR_5}\colon\pi_1(\sR_1\e)\isoto \pi_1(\sR_5\e)\isoto \pi_1(\sR_2\e)$.
We have a diagram of $t$-resolutions
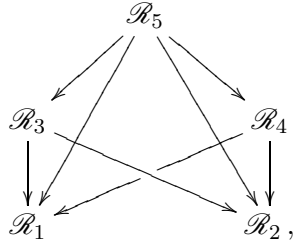
\[
\xymatrix{
& \sR_5\ar[dl]\ar[dr]\ar[ldd]\ar[rdd]\\
\sR_3\ar[d]\ar[drr] && \sR_4\ar[d]\ar[dll]|\hole \\
\sR_1 && \sR_2\, ,
}
\]
which may not commute. However, by Lemma \ref{lem:induced-hom}, this diagram
induces a {\em commutative} diagram of twisted constant $S$-group schemes and their isomorphisms
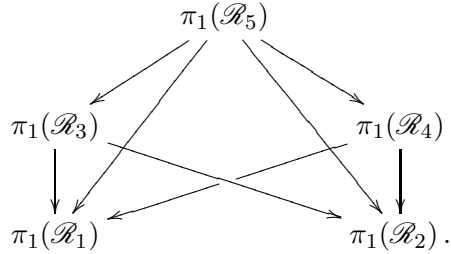
\[
\xymatrix{
& \pi_1(\sR_5)\ar[dl]\ar[dr]\ar[ldd]\ar[rdd]\\
\pi_1(\sR_3)\ar[d]\ar[drr] && \pi_1(\sR_4)\ar[d]\ar[dll]|\hole \\
\pi_1(\sR_1) && \pi_1(\sR_2)\, .
}
\]
We conclude that
\[
\psi_{\sR_3}=\psi_{\sR_5}=\psi_{\sR_4}\colon\pi_1(\sR_1\e)\isoto\pi_1(\sR_2\e),
\]
from which we deduce the existence of a {\em canonical} isomorphism $\psi\colon\pi_1(\sR_1\e)\isoto\pi_1(\sR_2\e)$.
\end{proof}

\begin{definition}\label{def:pi_1} Let $G$ be a reductive $S$-group scheme. Using
the preceding lemma, we shall henceforth identify the $S$-group schemes $\pi_{1}(\sR\e)$ as
$\s R$ ranges over the family of all
$t$-resolutions of $G$. Their common value will be denoted by
$\pi_1(G)$ and called the \emph{algebraic fundamental group of $G$}. Thus
\[
\pi_1(G)=\pi_1(\sR\e)
\]
for any $t$-resolution $\s R$ of $G$.
\end{definition}

Note that, by \eqref{t-r-p1}, a $t$-resolution $1\to T\to H\to G\to 1$ of $G$ induces an exact sequence
\begin{equation}\label{t-r-p1g}
1\to T_*\to (H^{\tor})_{*}\to\pi_1(G)\to 1.
\end{equation}
Further, by Proposition \ref{mu-p1-gtor}, there exists a canonical exact sequence
\begin{equation}\label{mu-p1g-gtor}
1\to\mu(-1)\to\pi_1(G)\to(G ^{\tor})_{*}\to 1.
\end{equation}

\begin{remark}\label{rem:m-res}
One can also define  $\pi_1(G)$ using $m$-resolutions.
By an {\em $m$-resolution of $G$} we mean a short exact sequence
\begin{equation}\label{e:m-res}
1\to M\to H\to G\to 1,\tag{$\sR$}
\end{equation}
where $H$ is a reductive $S$-group scheme such that $H^\der$ is simply connected,
and $M$ is an $S$-group scheme of multiplicative type.
Clearly, a $t$-resolution of $G$ is in particular an $m$-resolution of $G$.
It is very easy to see that any reductive $S$-group scheme $G$ admits an $m$-resolution:
we can take $H:=\rad(G)\times_S \Gtil$, with the homomorphism $H\to G$
from the beginning of the proof of Proposition \ref{ex-t-resol}, and set $M:=\mu_1=\krn[H\to G]$,
which is a  finite $S$-group scheme of multiplicative type.

Now let $\sR$ be an $m$-resolution of $G$ and consider the induced homomorphism $M\to H^\tor$.
We claim that there exists a complex of $S$-tori $T\to R$ which is isomorphic to $M\to H^\tor$ in the derived category.
Indeed, by \cite[Proposition B.3.8]{cnrd} there exists an embedding $M\into T$ of $M$ into an $S$-torus $T$.
Denote by $R$ the pushout of the homomorphisms $M\to H^\tor$ and $M\to T$.
Then the complex of $S$-tori $T\to R$ is quasi-isomorphic to the complex $M\to H^\tor$, as claimed.

Now we choose an $m$-resolution $\sR$ of $G$,
a complex of $S$-tori $T\to R$ which is isomorphic to $M\to H^\tor$ in the derived category,
and set $\pi_1(G)=\pi_1(\sR):=\cok[T_*\to R_*]$.
\end{remark}

\section{Functoriality and exactness of $\pi_1$}

In this section we show that $\pi_1$ is an exact covariant functor from the category of reductive $S$-group schemes
to the category of finitely generated twisted constant $S$-group schemes.

\begin{definition}\label{def:resol-of-morph}
Let $\vk\colon G_1\to G_2$ be a homomorphism of reductive $S$-group
schemes. A {\em $t$-resolution of $\vk$}, written $\vk_{\sR}\colon\sR_1\to\sR_2$, is an exact commutative diagram
\[
\xymatrix{(\sR_1)& 1\ar[r]   &T_1\ar[r]\ar[d]^{\vk_{\le T}}
&H_1\ar[r]\ar[d]^{\vk_{H}}   &G_1\ar[r]\ar[d]^\vk
 &1\\
 (\sR_2)& 1\ar[r]   &T_2\ar[r]                &H_2\ar[r] &G_2\ar[r] &1, }
\]
where $\sR_1$ and $\sR_2$ are $t$-resolutions of $G_1$ and $G_2$, respectively.

\end{definition}

Thus, if $G$ is a reductive $S$-group scheme and $\sR^{\e\prime}$ and $\sR$ are two $t$-resolutions of $G$,
then a morphism from $\sR^{\e\prime}$ to $\sR$ (as in Definition \ref{def morph-resol}) is a $t$-resolution of $\id_{G}\colon G\to G$.

\begin{remark}\label{rem:123} A $t$-resolution $\vk_{\sR}\colon\sR_1\to\sR_2$ of $\vk\colon G_1\to G_2$
induces a homomorphism of finitely generated twisted constant $S$-group schemes
\[
\pi_1(\vk_{\sR})\colon \pi_1(\sR_1)\to\pi_1(\sR_2).
\]
If $G_3$ is a third reductive $S$-group scheme, $\lambda\colon G_2\to G_3$ is an $S$-homomorphism
and $\lambda_{\sR}\colon\sR_2\to\sR_{\e 3}$ is a $t$-resolution of $\lambda$, then $\lambda_{\sR}\circ\vk_{\sR}\colon\sR_1\to\sR_{\e 3}$
is a $t$-resolution of $\lambda\circ\vk$ and
$$
\pi_1(\lambda_\sR\circ\vk_\sR)=\pi_1(\lambda_\sR)\circ\pi_1(\vk_\sR).
$$
\end{remark}

\begin{lemma}\label{lem:resol-of-morph}
Let  $\vk\colon G_1\to G_2$ be a homomorphism of reductive $S$-group
schemes and let $\sR_2$ be a $t$-resolution of $G_2$. Then there exists
a $t$-resolution $\vk_\sR\colon \sR_1\to\sR_2$ of $\vk$ for a suitable choice of $t$-resolution $\sR_1$ of $G_1$.
In particular, every homomorphism of reductive $S$-group schemes admits a $t$-resolution.
\end{lemma}
\begin{proof} Choose any $t$-resolution $\sR^{\e\prime}_{1}$ of $G_1$ and apply
Proposition \ref{prop:Kottwitz} to $\vk$, $\sR^{\e\prime}_{1}$ and $\sR_2$.
\end{proof}

\begin{definition}\label{morph-kappa-resol} Let $\vk\colon G_1\to G_2$ be a homomorphism of reductive $S$-group schemes
and let $\vk^{\e\prime}_\sR\colon \sR^{\e\prime}_{1}\to\sR^{\e\prime}_{2}$ and $\vk_\sR\colon \sR_1\to\sR_2$ be two $t$-resolutions of $\vk$.
A {\em morphism from $\vk^{\e\prime}_{\sR}$ to $\vk_{\sR}$}, written $\vk^{\e\prime}_{\sR}\to\vk_{\sR}$, is a commutative diagram
\[
\xymatrix{
\sR^{\e\prime}_{1}\ar[r]^{\vk^{\e\prime}_{\sR}}\ar[d]
               &\sR^{\e\prime}_{2}\ar[d]\\
\sR_{1} \ar[r]^{\vk_\sR}         &\sR_2, }
\]
where the left-hand vertical arrow is a $t$-resolution of $\id_{G_{1}}$
and the right-hand vertical arrow is a $t$-resolution of $\id_{G_{2}}$.
By a $t$-resolution {\em dominating} a $t$-resolution $\vk_{\sR}$ of $\vk$ we mean a $t$-resolution $\vk^{\e\prime}_{\sR}$ of $\vk$
admitting a morphism $\vk^{\e\prime}_{\sR}\to\vk_{\sR}$.
\end{definition}

\begin{lemma}\label{lem:dom-of-morph}
If $\vk_\sR\colon \sR_1\to\sR_2$ and $\vk^{\e\prime}_\sR\colon \sR^{\e\prime}_{1}\to\sR^{\e\prime}_{2}$ are  two $t$-resolutions
of a morphism $\vk\colon G_1\to G_2$, then there exists a third
$t$-resolution $\vk^{\e\prime\prime}_{\sR}$ of $\vk$ which dominates
both $\vk_{\sR}$ and $\vk^{\e\prime}_{\sR}$.
\end{lemma}
\begin{proof} By Corollary \ref{cor:domination}, there exists a $t$-resolution $\sR^\epp_2$ of $G_2$ which dominates both $\sR_2$ and $\sR^\ep_2$.
On the other hand, by Lemma \ref{lem:resol-of-morph}, there exists
a $t$-resolution $\widetilde{\vk}_\sR\colon \sR^\eppp_1\to\sR^\epp_2$ of $\vk$ for a suitable choice of $t$-resolution $\sR^\eppp_1$ of $G_1$.
Now a second application of Corollary \ref{cor:domination} yields a $t$-resolution $\sR^\epp_1$ of $G_1$
which dominates $\sR_1$, $\sR^\ep_1$ and $\sR^\eppp_1$.
Let $\phi\colon \sR^\epp_1\to \sR^\eppp_1$ be the corresponding morphism, which is a $t$-resolution of $\id_{G_{1}}$.
Then $\vk^\epp_\sR=\widetilde{\vk}_\sR\circ\phi\colon \sR^\epp_1\to\sR^\epp_2$ is a $t$-resolution of $\vk$ which
dominates both $\vk_\sR$ and $\vk^\ep_\sR$.
\end{proof}

\begin{construction}
Let  $\vk\colon G_1\to G_2$ be a homomorphism of reductive $S$-group
schemes. By Lemma \ref{lem:resol-of-morph}, there exists a
$t$-resolution $\vk_\sR\colon \sR_1\to\sR_2$ of  $\vk$, which induces a
homomorphism $\pi_1(\vk_\sR)\colon \pi_1(\sR_1)\to\pi_1(\sR_2)$ of finitely generated twisted constant $S$-group schemes.
Thus, if we identify $\pi_1(G_i)$ with $\pi_1(\sR_i)$ for $i=1,2$,
we obtain an $S$-homomorphism $\pi_1(\vk_\sR)\colon \pi_1(G_1)\to\pi_1(G_2)$ which, by Lemma \ref{lem:dom-of-morph},
can be shown to be independent of the chosen $t$-resolution $\vk_\sR$ of $\vk$. We denote it by
\[
\pi_1(\vk)\colon \pi_1(G_1)\to\pi_1(G_2).
\]
\end{construction}

\begin{lemma}\label{p1-comp}
Let $G_1\labelto{\vk} G_2\labelto{\lambda} G_3$ be homomorphisms of
reductive $S$-group schemes. Then
\[
\pi_1(\lambda\circ\vk)=\pi_1(\lambda)\circ\pi_1(\vk).
\]
\end{lemma}
\begin{proof} Choose a $t$-resolution $\sR_3$ of $G_3$. Applying Lemma
\ref{lem:resol-of-morph} first to $\lambda$ and then to
$\vk$, we obtain $t$-resolutions
$\sR_1\labelto{\vk_\sR} \sR_2\labelto{\lambda_\sR}\sR_3$ of $\vk$ and $\lambda$, and the composition $\lambda_\sR\circ\vk_\sR$
is a $t$-resolution of $\lambda\circ\vk$. Thus, by Remark \ref{rem:123},
\[
\pi_1(\lambda\circ\vk)=\pi_1(\lambda_\sR\circ\vk_\sR)=\pi_1(\lambda_\sR)\circ
\pi_1(\vk_\sR)=\pi_1(\lambda)\circ\pi_1(\vk),
\]
as claimed.
\end{proof}

Summarizing, for any non-empty scheme $S$, we have constructed a covariant functor $\pi_1$ from the category of reductive $S$-group schemes
to the category of finitely generated twisted constant $S$-group schemes.
Now assume that $S$ is admissible in the sense of \cite[Definition 2.1]{GA-flasque}
(i.e., reduced, connected, locally Noetherian and geometrically unibranch),
so that every reductive $S$-group scheme admits a flasque resolution \cite[Proposition 3.2]{GA-flasque}.
In this case the functor $\pi_{1}$ defined here in terms of $t$-resolutions 
coincides with the functor $\pi_1$ defined in \cite[Definition~3.7]{GA-flasque} in terms of flasque resolutions,
because a flasque resolution is a particular case of a $t$-resolution.
A basic example of a non-admissible scheme $S$ to which the constructions of the present paper apply, but not those of \cite{GA-flasque},
is an algebraic curve over a field having an ordinary double point. See \cite[Remark 2.3]{GA-flasque}.

\medskip

The following result generalizes \cite[Lemma 3.7]{BKG}, \cite[Proposition~6.8]{ct} and \cite[Theorem~3.14]{GA-flasque}.

\begin{theorem}\label{thm:pi1-exact}
Let $1\to G_1\to G_2\to G_3\to
1$ be an exact sequence of reductive $S$-group schemes. Then
the induced sequence of finitely generated twisted constant $S$-group schemes
\[
0\to \pi_1(G_1) \to \pi_1(G_2) \to \pi_1(G_3)\to 0
\]
is exact.
\end{theorem}
\begin{proof} The proof is similar to that of \cite[Theorem~3.14]{GA-flasque} using the exact sequence \eqref{mu-p1g-gtor}.
Namely, one first proves the theorem when $G_1$ is semisimple using the same arguments as in the proof of \cite[Lemma~3.12]{GA-flasque}
(those arguments rely on
\cite[Proposition~2.8]{GA-flasque}, which is valid over any non-empty base scheme $S$).
Secondly, one proves the theorem when $G_1$ is an $S$-torus using the same arguments
as in the proof of \cite[Lemma~3.13]{GA-flasque} (which rely on
\cite[Proposition~2.9]{GA-flasque}, which again holds over any non-empty base scheme $S$).
Finally, the theorem is obtained by combining these two particular cases as in the proof of \cite[Theorem~3.14]{GA-flasque}.
\end{proof}

We shall now present a second proof of Theorem \ref{thm:pi1-exact} which relies
on the \'etale-local existence of maximal tori in reductive $S$-group schemes.
To this end, we shall first show that if $G$ is a reductive $S$-group scheme which contains a maximal torus $T$,
then $T$ canonically determines a $t$-resolution of $G$.

\begin{lemma}\label{c:t-special}
Let $G$ be a reductive $S$-group scheme having a maximal $S$-torus $T$, and set $\Ttil:=\Gtil\times_{G}T$, it is a maximal $S$-torus of $\Gtil$.
Then there exists a $t$-resolution of $G$
\[
1\to \Ttil\to H\to G\to 1 \tag{$\sR_{\e T}\e$}
\]
such that $H^\tor$ is canonically isomorphic to $T$.
\end{lemma}
\begin{proof} By \cite[proof of Proposition 3.2]{GA-flasque}, the product in $G$ and the canonical epimorphism $\Gtil\to G^{\der}$
induce a faithfully flat homomorphism
$\rad(G\e)\times_{S}\e \Gtil\ra G$ whose (central) kernel $\mu_{1}$ embeds into $Z(\Gtil\,)$
via the canonical projection $\rad(G\e)\times_{S}\e \Gtil\to\Gtil$.
In particular, we have a central extension
\begin{equation}\label{e:G-push}
1\to\mu_{1}\overset{\varphi}\to\rad(G\e)\times_{S}\Gtil\to G \to 1.
\end{equation}
Since $Z(\Gtil\,)\subset\Ttil$ by \cite[Exp.~XXII, Corollary 4.1.7]{sga3}, we obtain an embedding $\psi\colon\mu_{1}\into\Ttil$.
Let $H$ be the pushout of
$\varphi\colon\mu_{1}\into
\rad(G)\times_{S}\Gtil$ and $\psi\colon\mu_{1}\into\Ttil$, i.e., the cokernel of the central embedding
\begin{equation}\label{e:H-push}
(\varphi,{\rm{inv}}_{\Ttil}\be\circ\be\psi\e)_{S}\colon\mu_{1} \into \big(\rad(G)\times_{S}\Gtil\e\big)\times_{S}  \Ttil.
\end{equation}
Now let $\varepsilon\colon S\to\rad(G)\times_{S}\Gtil$ be the
unit section of $\rad(G)\times_{S}\Gtil$ and set
$$
j=(\varepsilon,
{\rm{id}}_{\Ttil})_{S}\colon S\times_{S}\Ttil\to \big(\rad(G)\times_{S}\Gtil\e\big)\times_{S} \Ttil.
$$
Composing $j$ with the canonical isomorphism $\Ttil\simeq S\times_{S}\Ttil$, we obtain an $S$-morphism
$\Ttil\to \rad(G)\times_{S}\Gtil\times_{S}\e\Ttil$ which induces an embedding $\iota_{T}\colon\Ttil\into H$.
Further, let $\pi_{\e T}\colon H\to G$ be the homomorphism which is induced by the projection
\[
\rad(G)\times_{S}\Gtil\times_{S}  \Ttil \to \rad(G)\times_{S}  \widetilde{G}.
\]
Then we obtain a $t$-resolution of $G$
\[
1\longrightarrow \Ttil\overset{\iota_{_T}}\longrightarrow H\overset{\pi_{_T}}\longrightarrow G\longrightarrow 1
\]
which is canonically determined by $T$ (cf. the proof of Proposition \ref{ex-t-resol}).
It remains to show that $H^{\tor}$ is canonically isomorphic to $T$.
Let $\varepsilon_\rad\colon S\to \rad(G)$ and $\varepsilon_\Ttil\colon S\to\Ttil$ be the unit sections of $\rad(G)$ and $\Ttil$,
respectively, and consider the homomorphism
\[
(\varepsilon_\rad,{\rm{id}}_{\Gtil}\e,\varepsilon_\Ttil)_{S}\colon S\times_{S}
\Gtil\times_{S}S\to\rad(G)\times_{S}\Gtil\times_{S}  \Ttil.
\]
Composing this homomorphism with the canonical isomorphism $\Gtil\simeq S\times_{S}
\Gtil\times_{S}S$, we obtain a canonical embedding $\Gtil\into \rad(G)\times_{S}\Gtil\times_{S}\Ttil$.
The latter map induces a homomorphism $\Gtil\to H$ which identifies $\Gtil$ with $H^{\der}$.
Now consider the composite homomorphism
\[
\varphi_\rad\colon\mu_{1}\labelto{\varphi}\rad(G)\times_{S}\Gtil\overset{\rm pr_{_{\be 1}}}\longrightarrow\rad(G).
\]
Then $H^{\tor}:=H/H^\der=H/\Gtil$ is isomorphic to the cokernel of the central embedding
\begin{equation}\label{e:Ht-push}
(\varphi_{\rad},{\rm{inv}}_{\Ttil}\circ \psi\e)_{S}\colon\mu_{1}\into\rad(G)\times_{S}\Ttil.
\end{equation}
Compare \eqref{e:H-push}. Finally, the canonical embedding $\Ttil\into\Gtil$
induces an embedding $H^\tor\into G$ (see \eqref{e:G-push} and \eqref{e:Ht-push})
whose image is $\rad(G)\cdot (T\cap G^\der)=T$ \cite[Exp.~XXII, proof of Proposition 6.2.8(i)]{sga3}. This completes the proof.
\end{proof}

\begin{remark}
It is clear from the above proof that the homomorphism
$\Ttil\to H^{\tor}=T$ induced by the $t$-resolution $\sR_{\e T}\e$ of Lemma \ref{c:t-special}
is the canonical homomorphism $\partial\colon\Ttil\to T$.
\end{remark}

\begin{definition}\label{g,t}
Let $G$ be a reductive $S$-group scheme containing a maximal $S$-torus $T$.
The \emph{algebraic fundamental group of the pair $(G,T)$} is the $S$-group scheme $\pi_{1}(G,T):=\cok[\partial_{*}\colon \Ttil_{*}\to T_{*}]$.
\end{definition}

\def\ve{{\varepsilon}}
By Lemma \ref{c:t-special} and Definition \ref{def:pi_1}  we have a canonical isomorphism
\begin{equation}\label{p1gt=p1g}
\vartheta_T\colon \pi_1(G,T)\isoto\pi_1(\sR_T)=\pi_1(G).
\end{equation}
Further, any morphism of pairs $\vk\colon(G_1,T_1)\to (G_2,T_2)$ (in the obvious sense)
induces an $S$-homomorphism $\vk_{*}\colon\pi_1(G_1,T_1)\to \pi_1(G_2,T_2)$.
It can be shown that the following diagram commutes:
\begin{equation}\label{p1gt,p1g, morph}
\xymatrix{
\pi_1(G_1,T_1)\ar[r]^{\vk_*}\ar[d]_{\vartheta_{T_1}} &\pi_1(G_2,T_2)\ar[d]^{\vartheta_{T_2}}\\
\pi_1(G_1)\ar[r]^{\vk_*} &\pi_1(G_2)\, .
}
\end{equation}
This is immediate in the case where $\vk$ is a \emph{normal} homomorphism, i.e. $\vk(G_1)$ is normal in $G_2$
(this is the only case needed in this paper).
Indeed, in this case we have $\vk(\rad(G_1))\subset\rad(G_2)$ and therefore $\vk$ induces a morphism of $t$-resolutions
$\vk_\sR\colon \sR_{\e T_1}\to\sR_{\e T_2}$. See the proof of Lemma  \ref{c:t-special}.

\begin{remark} The preceding considerations and Lemma \ref{lem:can-isom} show that,
if $S$ is an admissible scheme in the sense of \cite[Definition 2.1]{GA-flasque},
so that every reductive $S$-group scheme $G$ admits a flasque resolution $\s F$,
and $G$ contains a maximal $S$-torus $T$, then there exists a canonical isomorphism
$\pi_{1}(\s F\e)\cong\cok[\partial_{*}\colon \Ttil_{*}\to T_{*}]$.
This fact generalizes \cite[Proposition A.2]{ct}, which is the case $S=\spec k$, where $k$ is a field, of the present remark.
\end{remark}

\begin{lemma}\label{lem:mt-exact}
Let
\[
1\to (G_1,T_1)\labelto{\vk} (G_2,T_2)\labelto{\lambda} (G_3,T_3)\to 1
\]
be an exact sequence of reductive $S$-group schemes with maximal tori.
Then the sequence of \'etale, finitely generated twisted constant $S$-group schemes
\[
0\to\pi_1(G_1,T_1) \labelto{\vk_*}\pi_1(G_2, T_2)  \labelto{\lambda_*}\pi_1(G_3,T_3) \to 0
\]
is exact.
\end{lemma}

\begin{proof} The assertion of the lemma is local for the \'etale topology, so
we may and do assume that $T_1$, $T_2$, and $T_3$ are split. By \cite[Proposition~2.10]{GA-flasque},
there exists an exact commutative diagram of reductive $S$-group schemes
\[
\xymatrix{ 1 \ar[r]  &\Gtil_1\ar[r]\ar[d]^{\partial_{\e 1}}
 &\Gtil_2\ar[r]\ar[d]^{\partial_{\e 2}}  &\Gtil_3\ar[r]\ar[d]^{\partial_{\e 3}} &1\\
1 \ar[r]  &G_1\ar[r]            &G_2\ar[r]            &G_3\ar[r] &1 ,
}
\]
which induces an exact commutative diagram of split $S$-tori
\begin{equation}\label{e:d-T}
\xymatrix{
1 \ar[r]  &\Ttil_1\ar[r]\ar[d]^{\partial_{\e 1}}  &\Ttil_2\ar[r]\ar[d]^{\partial_{\e 2}}  &\Ttil_3\ar[r]\ar[d]^{\partial_{\e 3}} &1\\
1 \ar[r]  &T_1\ar[r]            &T_2\ar[r]            &T_3\ar[r] &1,
}
\end{equation}
where $\Ttil_i:=\Gtil_i\times_{G_i} T_i$ $(i=1,2,3)$.
Now, as in \cite[Proof of Lemma 3.7]{BKG},  diagram \eqref{e:d-T} induces an
exact commutative diagram of constant $S$-group schemes
\begin{equation*}
\xymatrix{
1 \ar[r]  &\Ttil_{1*_{\phantom{}}}\ar[r]\ar@{^{(}->}[d]^{\partial_{\e 1*}}
       &\Ttil_{2*_{\phantom{}}}\ar[r]\ar@{^{(}->}[d]^{\partial_{\e 2*}}  &\Ttil_{3*_{\phantom{}}}\ar[r]\ar@{^{(}->}[d]^{\partial_{\e 3*}} &1\\
1 \ar[r]  &T_{1*}\ar[r]            &T_{2*}\ar[r] &T_{3*}\ar[r] &1 }
\end{equation*}
with injective vertical arrows. An application of the snake lemma to
the last diagram now yields the exact sequence
\[
0\to\cok \,\partial_{\e 1*} \to\cok \,\partial_{\e 2*} \to
\cok\,\partial_{\e 3*} \to 0,
\]
which is the assertion of the lemma.
\end{proof}

\begin{proof}[Second proof of Theorem \ref{thm:pi1-exact}]
Let $1\to G_1\to G_2\to G_3\to 1$ be an exact sequence of reductive
$S$-group schemes. By \cite[Exp.~XIX, Proposition 6.1]{sga3}, for any reductive $S$-group scheme $G$
there exists an \'etale covering $\{S_{\alpha}\to S\}_{\alpha\in A}$
such that each $G_{\be S_\alpha}:=G\times_S S_\alpha$ contains a split
maximal $S_\alpha$-torus $T_\alpha$. Thus, since the assertion of the theorem is local for the \'etale topology,
we may and do assume that $G_2$ contains a split maximal $S$-torus $T_2$. Let $T_1=G_{1}\times_{G_2}T_{2}$ and
let $T_3$ be the cokernel of $T_1\to T_2$. Then $T_i$ is a split maximal $S$-torus of $G_i$ for $i=1,2,3$
and we have an exact sequence of pairs
\[
1\to (G_1, T_1) \to (G_2, T_2) \to (G_3, T_3) \to 1.
\]
Now the theorem follows from Lemma \ref{lem:mt-exact}, \eqref{p1gt=p1g} and \eqref{p1gt,p1g, morph}.
\end{proof}

\section{Abelian cohomology and $t$-resolutions}

Let $S_{\e\rm{fl}}$ (respectively, $S_{\e\rm{\acute{e}t}}$) be the small fppf
(respectively, \'etale) site over $S$. If $F_{1}$ and $F_{2}$ are
abelian sheaves on $S_{\e\rm{fl}}$ (regarded as complexes
concentrated in degree 0), $F_{1}\lbe\otimes^{\e\mathbf{L}}\lbe
F_{2}$ (respectively, $\rhom\e(F_{1},F_{2})$) will denote the total
tensor product (respectively, right derived Hom functor) of $F_{1}$
and $F_{2}$ in the derived category of the category of abelian
sheaves on $S_{\e\rm{fl}}$.

Let $G$ be a reductive group scheme over $S$. For any integer $i\geq -1$, the $i$-th
\textit{abelian (flat) cohomology group of $G$} is by definition the
hypercohomology group
\[
H^{\le i}_{\rm{ab}}(S_{\e\rm{fl}}, G\e)={\bh}^{\e
i}\big(S_{\e\rm{fl}},
Z\big(\Gtil\,\big)\overset{\partial_{Z}}\longrightarrow Z(G\e)).
\]
On the other hand, the $i$-th \textit{dual abelian cohomology group of $\e G$} is the group
\[
H^{\le i}_{\rm{ab}}(S_{\e\rm{\acute{e}t}},G^{*})={\bh}^{\e
i}\big(S_{\e\rm{\acute{e}t}},Z(G\e)^{*}\!\overset{\partial_{\lbe
Z}^{*}}\longrightarrow Z\big(\Gtil\,\big)^{\lbe *}\e\big).
\]
Here all the complexes of length 2 are in degrees $(-1,0)$.
See \cite[beginning of \S4]{GA-flasque} for basic properties of these cohomology groups and
\cite{bor,ga2,ga3} for (some of) their arithmetical applications.

The following result is an immediate consequence of \eqref{q-iso}.

\begin{proposition}\label{prop-cohom}
Let $G$ be a reductive $S$-group scheme and let $1\ra T\ra H\ra G\ra 1$ be a $t$-resolution of $G$. Then the given $t$-resolution
defines isomorphisms $H^{\le i}_{\rm{ab}}(S_{\e\rm{fl}}, G\e)\simeq
{\bh}^{\e i}\big(S_{\e\rm{fl}},T\ra R\e)$ and $H^{\le
i}_{\rm{ab}}(S_{\e\rm{\acute{e}t}}, G^{*}\e)\simeq {\bh}^{\e
i}\big(S_{\e\rm{\acute{e}t}},R^{*}\ra T^{*})$, where $R=H^{\lbe\tor}$. Further, there exist exact sequences
$$
\dots\ra H^{\le i}(S_{\e\rm{\acute{e}t}},T\e) \ra H^{\le
i}(S_{\e\rm{\acute{e}t}},R\e)\ra H^{\le
i}_{\rm{ab}}(S_{\e\rm{fl}},G\e)\ra H^{\le
i+1}(S_{\e\rm{\acute{e}t}},T\e)\ra\dots
$$
and
$$
\dots\ra H^{\le i}(S_{\e\rm{\acute{e}t}},R^{*}\e) \ra H^{\le
i}(S_{\e\rm{\acute{e}t}},T^{*}\e)\ra H^{\le
i}_{\rm{ab}}(S_{\e\rm{\acute{e}t}},G^{*}\e)\ra H^{\le
i+1}(S_{\e\rm{\acute{e}t}},R^{*}\e)\ra\dots.\qed
$$
\end{proposition}

\begin{corollary}\label{cor-pi1} Let $G$ be a reductive $S$-group scheme.
Then, for every integer $i\geq -1$, there exist isomorphisms
$$
H^{\le i}_{\rm{ab}}(S_{\e\rm{fl}}, G \e)\simeq{\bh}^{\e i}
(S_{\e\rm{fl}},\pi_{1}\lbe(G\e)\lbe\otimes^{\e\mathbf{L}}\be
\bg_{m,S})
$$
and
$$
H^{\le i}_{\rm{ab}}(S_{\e\rm{\acute{e}t}},G^{*}\e)\simeq{\bh}^{\e i}
(S_{\e\rm{\acute{e}t}},\rhom(\pi_{1}\lbe(G\e),\bz_{\! S})).
$$
\end{corollary}
\begin{proof} This follows from  Proposition \ref{prop-cohom}  in the same way
as \cite[Corollary 4.3]{GA-flasque} follows from \cite[Proposition 4.2]{GA-flasque}.
\end{proof}

\begin{proposition} Let $1\ra G_{1}\ra G_{2}\ra G_{3}\ra 1$ be an exact sequence of reductive $S$-group schemes.
Then there exist exact sequences of abelian groups
$$
\dots\ra H^{\le i}_{\rm{ab}}(S_{\e\rm{fl}},G_{1}\e)\ra H^{\le
i}_{\rm{ab}}(S_{\e\rm{fl}},G_{2}\e)\ra H^{\le
i}_{\rm{ab}}(S_{\e\rm{fl}}, G_{3}\e)\ra H^{\le
i+1}_{\rm{ab}}(S_{\e\rm{fl}},G_{1}\e)\ra\dots
$$
and
$$
\dots\ra H^{\le i}_{\rm{ab}}(S_{\e\rm{\acute{e}t}},G_{3}^{*}\e)\ra
H^{\le i}_{\rm{ab}}(S_{\e\rm{\acute{e}t}},G_{2}^{*}\e)\ra H^{\le
i}_{\rm{ab}}(S_{\e\rm{\acute{e}t}}, G_{1}^{*}\e)\ra H^{\le
i+1}_{\rm{ab}}(S_{\e\rm{\acute{e}t}},G_{3}^{*}\e)\ra\dots.
$$
\end{proposition}
\begin{proof} This follows from  Corollary \ref{cor-pi1} and Theorem \ref{thm:pi1-exact}.
\end{proof}

\bigskip\noindent
{\bf Acknowledgements.}
M.~Borovoi was partially supported by the Hermann Minkowski Center for Geometry.
C.D.~Gonz\'alez-Avil\'es was partially supported by Fondecyt grant 1120003.
The authors are very grateful to Brian Conrad for proving \cite[Proposition B.3.8]{cnrd},
which we used in the proof of Proposition \ref{ex-t-resol} 
and in a construction in Remark \ref{rem:m-res}.
We are  grateful to Joseph Bernstein for his help in proving Lemma \ref{lem:Bernstein},
and to Jean-Louis Colliot-Th\'el\`ene for most helpful discussions.
We thank the anonymous referees for their helpful remarks.
This paper was completed during a stay of both authors at the Max-Planck-Institut f\"ur Mathematik, Bonn,
and we are very grateful to this institute for hospitality, support and excellent working conditions.

\end{document}